\documentclass[11pt]{amsart}
\usepackage{amsmath}
\usepackage{amssymb}
\usepackage{amscd}

\def\NZQ{\mathbb}               % the font for N,Z,Q,R,C
\def\NN{{\NZQ N}}
\def\QQ{{\NZQ Q}}

\def\RR{{\NZQ R}}

\newtheorem{Theorem}{Theorem}[section]
\newtheorem{Lemma}[Theorem]{Lemma}
\newtheorem{Corollary}[Theorem]{Corollary}

\newtheorem{Example}[Theorem]{Example}

\let\epsilon\varepsilon
\let\phi=\varphi
\let\kappa=\varkappa

\begin{document}

\title{Semigroups of Valuations Dominating Local Domains}

\author{Steven Dale Cutkosky}
\thanks{The  author was partially supported by NSF }

\maketitle

%\begin{abstract}
%\end{abstract}
Let $(R,m_R)$ be an equicharacteristic local domain, with quotient field $K$. Suppose that $\nu$ is
a valuation of $K$ with valuation ring $(V,m_V)$. Suppose that $\nu$ dominates $R$; that is,
$R\subset V$ and $m_V\cap R=m_R$. The possible value groups $\Gamma$ of $\nu$ have been extensively
studied and classified, including  in the papers MacLane \cite{M}, MacLane and Schilling \cite{MS}, Zariski and Samuel \cite{ZS}, and Kuhlmann \cite{K}.
The most basic fact is that there is an order preserving
embedding of $\Gamma$ into $\RR^n$ with the lex order, where $n$ is the dimension of $R$. 

The
semigroup
$$
S^R(\nu)=\{\nu(f)\mid f\in m_R-\{0\}|
$$
is however not well understood, although it is known to encode important information about the
topology, resolution of singularities and ideal theory of $R$.

In Zariski and Samuel's classic book on Commutative Algebra \cite{ZS}, two general facts about the
semigroup $S^R(\nu)$ are proven (Appendix 3 to Volume II).
\begin{enumerate}
\item[1.] $S^R(\nu)$ is a well ordered subset of the positive part of the value group $\Gamma$ of
$\nu$ of ordinal type at most $\omega^h$, where $\omega$ is the ordinal type of the well ordered
set $\NN$, and $h$ is the rank of the valuation. \item[2.] The rational rank of $\nu$ plus the
transcendence degree of $V/m_V$ over $R/m_R$ is less than or equal to the dimension of $R$.
\end{enumerate}
The second condition is the Abhyankar inequality \cite{Ab}.

Prior to this paper, no other general constraints were known on  value  semigroups
$S^R(\nu)$. In fact, it was even unknown if the above conditions 1 and 2 characterize value
semigroups.

In this paper, we construct an example of a well ordered subsemigroup of
$\QQ_+$ of ordinal type $\omega$, which is not a value semigroup of an equicharacteristic
noetherian local domain. This  shows that the above conditions 1 and 2 do not characterize value
semigroups on equicharacteristic noetherian local domains. We constuct this  in  Example \ref{Example2} by finding a
new constraint, Theorem \ref{Theorem1}, on a semigroup being a value semigroup of an
equidimensional noetherian local domain of dimension $n$. In Corollary \ref{Corollary1}, we give a
stronger constraint on regular local rings.

In \cite{CT}, Teissier and the author give some examples showing that some surprising semigroups of
rank $>1$ can occur as   semigroups of valuations on noetherian domains, and raise the general
question of finding new constraints on
value semigroups and classifying semigroups which occur as value semigroups. 

The only semigroups which are realized by a valuation on a one dimensional regular local ring are
isomorphic to the natural numbers. The semigroups which are realized by a valuation on a regular
local ring of dimension 2 with algebraically closed residue field are much more complicated, but are completely classified by Spivakovsky in \cite{S}. A different proof is given by  Favre and Jonsson in \cite{FJ}, and the theorem is formulated in the  context of semigroups by Cutkosky and Teissier \cite{CT}.
However, very little is known in higher dimensions. The classification of semigroups of valuations
on regular local rings of dimension two does suggest that there may be constraints on the rate of
growth of the number of new generators on semigroups of valuations dominating a noetherian domain.
We prove that there is such a constraint, giving a new necessary condition for a semigroup to be a
value semigroup. This is accomplished in Theorem \ref{Theorem1} and Corollary \ref{Corollary1}. The
constraint is  far from sharp, even for one dimensional regular local rings. The criterion
of Theorem \ref{Theorem1} is however sufficiently strong to allow us to give a very simple example,
Example \ref{Example2}, of a well ordered subsemigroup $S$ of $\QQ_+$ of ordinal type $\omega$ which is
not the semigroup of a valuation dominating an equidimensional noetherian local domain.

\section{Semigroups of valuations dominating polynomial rings}
Given a set $T$, $|T|$ will denote the cardinality of $T$. If $S$ is a totally ordered abelian semigroup,
$S_+$ will denote the set of positive elements of $S$.

Let $k$ be a field, $n$ be a positive integer, and $A_n=k[x_1,\ldots,x_n]$ be a polynomial ring in
$n$ variables. Let $m_n$ be the maximal ideal $m_n=(x_1,\ldots,x_n)$ of $A_n$. $A_n$ is naturally a
$k$ vector space. Let $F_{n,d}$ be the $k$ subspace of $m_n$ of polynomials of degree $\le d$ with
0 as the constant term.

 Suppose that $\nu$ is a valuation  of the quotient field of $A_n$
with valuation ring $V$. Suppose $A_n\subset V$ and $m_V\cap A_n= m_n$. Let $\Gamma$ be the value
group of $\nu$. Let $\Psi$ be the convex subgroup of real rank 1 of $\Gamma$, so that $\Psi$ is order isomorphic to a subgroup of $\RR$. Let
$S(\nu)=\nu(m_n-\{0\})\cap \Psi$.

If $S(\nu)\ne \emptyset$, then $S(\nu)$ is a well ordered semigroup of ordinal type $\omega$, by the first condition stated in the introduction, since $S(\nu)$ is a value semigroup on the quotient $A/\{f\in A\mid \nu(f)\not\in\Psi\}$.
 Let the smallest   element of $S(\nu)$ be $s_0$. Let
$S_d(\nu)=\nu(F_{d,n}-\{0\})$.
\begin{Lemma}\label{Lemma1} The upper bound
$$
|S_d(\nu)|< \binom{n+d}{n}
$$
holds for all $d\in\NN$.
\end{Lemma}
\begin{proof}
For $\lambda\in \Gamma$, the set $C_{\lambda}=\{f\in F_{d,n}\mid \nu(f)> \lambda\}$ is a $k$
subspace of $F_{d,n}$. Since $S_d(\nu)$ is well ordered and $\mbox{dim}_kF_{d,n}=\binom{n+d}{d}-1$,
the lemma follows.
\end{proof}

For $a,b\in\Psi$, let $[a,b)=\{c\in \Psi\mid a\le c <b\}$.

\begin{Lemma}\label{Lemma2}
For all $d\in\NN$,
$$
|S(\nu)\cap [0,(d+1)s_0)|<\binom{n+d}{n}.
$$
\end{Lemma}

\begin{proof}
Since every element of $m_n$ has value $\ge s_0$, every element of $m_n^{d+1}$ has value $\ge
(d+1)s_0$. For $f\in m_n$, we have an expression $f=g+h$ with $g\in F_{d,n}$ and $h\in m_n^{d+1}$. If $\nu(f)<(d+1)s_0$, we must have $\nu(f)=\nu(g)$. Thus $S(\nu)\cap [0,(d+1)s_0)=S_d(\nu)\cap [0,(d+1)s_0)$.
\end{proof}

\begin{Example}\label{Example1} There exists a well ordered subsemigroup $U$ of $\QQ_+$
such that $U$  has ordinal type $\omega$ and $U\ne S(\nu)$ for any valuation $\nu$  as above.
\end{Example}

\begin{proof}  Let $T$ be any subset of $\QQ_+$ such that
1 is the smallest element of $T$ and
$$
|T\cap [n,n+1)|  = \binom{2n}{n}
$$
 for all $n\ge 1$.  For all positive integers $r$, let
$rT=\{a_1+\cdots+a_r\mid a_1,\ldots,a_r\in T\}$.
Let
$U=\omega T=\cup_{r=1}^{\infty}rT$ be the semigroup generated by $T$.
By our construction, $|U\cap [0,r)|<\infty$ for all $r\in \NN$. Thus
$U$ is well ordered and has ordinal type $\omega$.

Suppose that there exists a field $k$, a positive integer $n$, and a valuation $\nu$ such that
$S(\nu)=U$. We have 
$$
\binom{2n}{n}= |T\cap [n,n+1)|\le |U\cap[n,n+1)|\le |U\cap[0,n+1)|\le\binom{2n}{n}-1,
$$
where the last inequality is by Lemma \ref{Lemma2},
which is a contradiction.
\end{proof}

\section{Semigroups of valuations dominating local domains}

Let $\nu$ be a valuation of a field $K$. Let $V$ be the valuation ring of $\nu$, and let $m_V$ be
the maximal ideal of $V$. Suppose that $R$ is a noetherian local domain with quotient field $K$
such that $\nu$ dominates $R$; that is, $R\subset V$ and $m_V\cap R=m_R$ is the maximal ideal of
$R$.  Let $\Gamma$ be the value group of $\nu$. $\Gamma$ has finite rank, since $\nu$ dominates the
noetherian local ring $R$, and by the second condition stated in the introduction. Let $\Psi$ be the convex subgroup of real rank 1 of $\Gamma$. Define a
semigroup of real rank 1 by
$$
S^R(\nu)=\nu(m_R-\{0\})\cap \Psi.
$$
For $a,b\in\Psi$, let $[a,b)=\{c\in\Psi\mid a\le c<b\}$.

 Let $\hat R$ be the $m_R$-adic completion
of $R$, which naturally contains $R$.  Then there exists a
prime ideal $Q$ of $\hat R$ such that  the valuation $\nu$
induces uniquely  a valuation $\overline \nu$ of the quotient field of $\hat R/Q$ which dominates
$\hat R/Q$, $\overline \nu$ has rank 1, $S^{\hat R/Q}(\overline \nu)=S^R(\nu)$, and
$s_0=\mbox{min}\{\overline \nu(f)\mid f\in m_{\hat R/Q}\}$. We will call $Q$ the prime ideal of
elements of $\hat A$ of infinite value. The construction of $Q$ and related considerations are discussed in Spivakovsky \cite{S}, Heinzer and Sally \cite{HS}, Cutkosky \cite{C}, Cutkosky and Ghezzi \cite{CG} and Teissier \cite{T}. We will outline the construction of $Q$ in the following
paragraph.

Let $\alpha=\mbox{min}\{\nu(g)\mid g\in m_R\}$.  Suppose that $f\in \hat R$. Then for any Cauchy
sequence $\{g_i\}$ in  $R$ which converges to $f$, exactly one of the following two conditions must
hold.
\begin{enumerate}
\item[1)] There exists $n\in\NN$ such that $\nu(g_i)\le n\alpha$ for all $i\gg 0$ or

\item[2)] For all $n\in \NN$, there exists $i\in\NN$ such that $\nu(g_i)>n\alpha$.
\end{enumerate}
Moreover, either condition 1) holds for all Cauchy sequences $\{g_i\}$ in $R$ which converge to $f$
or condition 2) hold for all Cauchy sequences $\{g_i\}$ in $R$ which converge to $f$. If condition
1) holds for $f$, then there exists $\lambda\in S^R(\nu)$ such that for any Cauchy sequence
$\{g_i\}$ in $R$ which converges to $f$, there exists $n\in\NN$ (which depends on the Cauchy
sequence) such that $\nu(g_i)=\lambda$ for all $i\ge n$. In this, way may extend $\nu$ to a
function on $\hat R$, defining for $f\in\hat R$, $\nu(f)=\lambda$ if 1) holds, and $\nu(f)=\infty$ if 2) holds. Let
$$
Q=\{f\in\hat R\mid \nu(f)=\infty\}.
$$
$Q$ is a prime ideal in $\hat R$.  We extend $\nu$
to a valuation $\overline \nu$ of the quotient field of $\hat R/Q$ with the desired properties by
defining $\overline\nu(f+Q)=\nu(f)$ for $f\in \hat R$.
If $\Gamma$ has rank 1, we have $Q\cap R=(0)$, and we have an inclusion $R\rightarrow \hat R/Q$.

\begin{Theorem}\label{Theorem1} Suppose that $A$ is an equicharacteristic local domain, and $\nu$ is a valuation
which dominates $A$. Let $s_0=\mbox{min}\{\nu(f)\mid f\in m_A\}$ and
$n=\mbox{dim}_{A/m_A}m_A/m_A^2$. Then
\begin{equation}\label{eq6}
|S^A(\nu)\cap [0,(d+1)s_0)|<\binom{n+d}{n}
\end{equation}
for all $d\in\NN$.
\end{Theorem}

\begin{proof}
Let $\hat A$ be the $m_A$-adic completion of $A$, with natural inclusion $A\rightarrow \hat A$. Let
$Q$ be the prime ideal of elements of $\hat A$ of infinite value, and let $\overline\nu$ be the
valuation of  the quotient field of $\hat A/Q$ induced by $\nu$ which dominates $\hat A/Q$, so that $S^{\hat
A/Q}(\overline\nu)=S^A(\nu)$, and $s_0=\mbox{min}\{\overline\nu(f)\mid f\in m_{\hat A/Q}\}$. Let
$k$ be a coefficient field of $\hat A/Q$. There exists a power series  ring
$R=k[[x_1,\ldots,x_n]]$, and a surjective $k$-algebra homomorphism $\phi:R\rightarrow \hat A/Q$.
Let $P$ be the kernel of $\phi$, so that $R/P\cong \hat A/Q$. We may thus identify $\overline \nu$
with  a rank 1 valuation of the quotient field of $R/P$ dominating $R/P$ such that
$S^{R/P}(\overline\nu)=S^{\hat A/Q}(\overline\nu)$, and $s_0=\mbox{min}\{\overline\nu(f)\mid f\in
m_{R/P}\}$.

Let $\nu_1$ be  a valuation of the quotient field of $R$ which dominates $R_P$, and let $\mu$ be
the composite valuation of $\nu_1$ and $\overline \nu$.  For $f\in R$, we have that $\nu_1(f)=0$ if
and only if $f\not\in P$, and
$$
S^R(\mu)=\mu(m_R-P)=\overline \nu(m_{R/P}-\{0\})=S^{R/P}(\overline\nu).
 $$
 We
 also have that $s_0=\mbox{min}\{\mu(f)\mid f\in m_R\}$.

$R$ contains the polynomial ring $A_n=k[x_1,\ldots,x_n]$. Suppose that $f\in m_R-P$. There exists a
Cauchy sequence $\{g_i\}$ in $A_n$ (for the $m_R$-adic topology) that converges to $f$. Since
$\overline\nu$ has rank 1, and $\nu_1(f)=0$, there exists $l\in\NN$ such that $\mu(h)>\mu(f)$
whenever $h\in m_R^l$. Further, since $\{g_i\}$ is a Cauchy sequence, there exists $e\in\NN$ such
that $f-g_i\in m_R^l$ whenever $i\ge e$.  Thus $\mu(f)=\mu(g_e)$.  It follows that
$S^{(A_n)_{m_n}}(\mu')=S^R(\mu)=S^A(\nu)$ where $\mu'$ is the restriction of $\mu$ to $A_n$, and
$s_0=\mbox{min}\{\mu'(f)\mid f\in m_n\}$. The conclusions of the theorem now follow from Lemma
\ref{Lemma2}.
\end{proof}

 When $\nu$ has
rank 1, $S^A(\nu)=\nu(m_A-\{0\})$ is the semigroup of $\nu$ on $A$.  Thus (\ref{eq6}) gives us a
necessary condition for a subsemigroup $S$ of $\RR_+$ to be the  semigroup of a valuation on an
equicharacteristic local domain $R$.

When $A$ is a regular local ring, the $n$ in the theorem   is the dimension of $A$. Thus we have
the following Corollary, which gives a necessary condition for a subsemigroup of $\RR_+$ to be the
subsemigroup of a valution on a equidimensional regular local ring.

\begin{Corollary}\label{Corollary1} Suppose that $A$ is an equicharacteristic regular local ring of dimension $n$,
and $\nu$ is a rank 1 valuation which dominates $A$. Let $s_0=\mbox{min}\{\nu(f)\mid f\in m_ A\}$.
Then \begin{equation}\label{eq5}
 |\nu(m_A-\{0\})\cap[0,(d+1)s_0)|<\binom{n+d}{n}
\end{equation}
for all $d\in \NN$.
\end{Corollary}

The necessary condition  (\ref{eq5}) is  far from giving  a sufficient condition for a subsemigroup
of $\RR_+$ to be the  semigroup of a valuation on an equidimensional  regular local ring  $A$, even
in the case when $A$ has dimension $n=1$.  Suppose that $\nu$ is a valuation dominating an
equidimensional  regular local ring $A$ of dimension $n=1$. $\nu$ must then be equivalent to the
$m_A$-adic valuation of $A$, and we thus have that $\nu(m_A-\{0\})=s_0\NN$. Thus
$$
|\nu(m_A-\{0\})\cap[0,(d+1)s_0)|=d
$$
for all $d\ge 1$, which is significantly less than the bound $\binom{1+d}{1}=d^2+d$ of the
corollary.

The criterion of the theorem is strong enough to give us a simple method of constructing well
ordered semigroups of ordinal type $\omega$ which are not value semigroups. The existence of such
examples was not previously known.

\begin{Example}\label{Example2} There exists a well ordered subsemigroup $U$ of $\QQ_+$ such that $U$ has ordinal type $\omega$ and $U\ne\nu(m_A-\{0\})$
 for any valuation $\nu$ dominating an equicharacteristic  noetherian local domain $A$.
\end{Example}

\begin{proof} Let $U$ be the semigroup constructed in Example \ref{Example1}. By Theorem
\ref{Theorem1} and the proof of Example \ref{Example1}, $U$ cannot be a semigroup of a valuation dominating an equicharacteristic  local
domain.
\end{proof}

\end{document}